\documentclass[a4paper,12pt]{amsart}
\usepackage{}
\usepackage{mathrsfs}

\usepackage{amssymb}
\usepackage{latexsym}
\usepackage{amsfonts}
\usepackage{amsmath}
\usepackage{eucal}
\usepackage{bm}
\usepackage{bbm}
\usepackage{graphicx}
\usepackage[english]{varioref}
\usepackage[nice]{nicefrac}
\usepackage[all]{xy}
\usepackage{amsthm}

\newcommand{\supp}{\text {\rm supp}}

\newcommand{\End}{{\rm End}}

\def\ds{\dot{s}}

\def\i{^{-1}}
\def\ge{\geqslant}
\def\le{\leqslant}
\def\<{\langle}
\def\>{\rangle}

\def\tPhi{\tilde{\Phi}}

\def\N{{\rm{N}}}

\def\a{\alpha}
\def\b{\beta}
\def\g{\gamma}

\def\D{\Delta}

\def\e{\epsilon}

\def\s{\sigma}

\def\k{\kappa}
\def\l{\lambda}

\def\Om{\Omega}

\def\ZZ{\mathbb Z}
\def\NN{\mathbb N}

\def\FF{\mathbb F}
\def\RR{\mathbb R}

\def\LL{\mathbb L}

\def\ca{\mathcal A}

\def\co{\mathcal O}
\def\cp{\mathcal P}

\def\cu{\mathcal U}

\def\tW{\tilde W}
\def\tw{\tilde w}

\theoremstyle{plain}
\newtheorem{thm}{Theorem}[section]
\newtheorem*{thm*}{Theorem}
\newtheorem{prop}[thm]{Proposition}
\newtheorem{lem}[thm]{Lemma}
\newtheorem{cor}[thm]{Corollary}

\theoremstyle{definition}

\theoremstyle{remark}

\newtheorem*{rmk}{Remark}
\newtheorem*{claim*}{Claim}

\begin{document}

\author{Sian Nie}
\thanks{{\it Address}: Institute of Mathematics, CAS, 100190, Beijing}
\thanks{{\it Current address}: Max-Planck Institute of Mathematics, 53111, Bonn}
\thanks{{\it E-mail}: niesian@amss.ac.cn}

\title[]{Fundamental elements of an affine Weyl group}

\begin{abstract} Fundamental elements are certain special elements of affine Weyl groups introduced by G\"{o}rtz, Haines, Kottwitz and Reuman. They play an important role in the study of affine Deligne-Lusztig varieties. In this paper, we obtain characterizations of the fundamental elements and their natural generalizations. We also derive an inverse to a version of ``Newton-Hodge decomposition" in affine flag varieties. As an application, we obtain a group-theoretic generalization of Oort's results on minimal $p$-divisible groups, and we show that, in certain good reduction reduction of PEL Shimura datum, each Newton stratum contains a minimal Ekedahl-Oort stratum. This generalizes a result of Viehmann and Wedhorn.
\end{abstract}
\maketitle

\section{Introduction}
Let $D$ be a $p$-divisible group over an algebraically closed field of positive characteristic. Following Oort, we say $D$ is {\it minimal} if for any $p$-divisible group $D'$ one has $$D[p] \cong D'[p] \Leftrightarrow D \cong D'.$$ Oort proved that in each isogeny class of $p$-divisible groups there exists a unique isomorphism class of minimal $p$-divisible groups \cite{O}. By taking its group-theoretic analogue, Viehmann and Wedhorn constructed {\it minimal Ekedahl-Oort strata} within a given {\it Newton stratum} for good reductions of PEL Shimura varieties, which led to a proof of the ``Manin's conjecture", i.e., the nonemptiness of all Newton strata for the PEL-case \cite{VW}. In their construction, the {\it fundamental ($P$-fundamental) alcove elements}, introduced by G\"{o}rtz, Haines, Kottwitz and Reuman \cite{GHKR}, played an important role in providing minimal representatives in each Newton stratum. The main purpose of this paper is to provide a detailed study of fundamental elements and their relations with minimal Ekedahl-Oort strata using group-theoretic methods.

\

Let $\FF_q$ be the finite field of $q$ elements and $\FF$ an algebraic closure of $\FF_q$. We denote by $\LL=\FF((\e))$ the field of formal power series over $\FF$. Let $G$ a (connected) reductive group over $\FF_q[[\e]]$.\footnote{We can also consider the reductive group defined over the ring of Witt vectors over $\FF$. Since we use group-theoretic methods, all the results of the paper remain true.} Then $G$ is quasi-split and splits over a finite unramified extension of $\FF_q[[\e]]$. Fix $T \subset B \subset G$, where $T$ (resp. $B$) is a maximal torus (resp. Borel subgroup) over $\FF_q[[\e]]$.  Let $I$ be the preimage of $B(\FF)$ under the natural projection $K=G(\FF[[\e]]) \to G(\FF)$, which is called an {\it Iwahori subgroup} of $G(\LL)$. By Bruhat decomposition, we have $\tW \cong I \backslash G(\LL) / I $. Here $\tW=\N_G(T)(\LL) / T(\FF[[\e]])$ denotes the {\it Iwahori-Weyl group} of $G$.

Let $\s$ be the Frobenius automorphism of $\LL$ over $\FF_q((\e))$. We also denote by $\s$ the induces automorphism on $G(\LL)$. Let $L \subset G$ is a Levi subgroup containing $T$. We say $y \in \tW$ is {\it $L$-permissible} if there exists a cocharacter $\l$ of $T$ satisfying $L=M_\l$ and $\bar y\s(\l)=\l$. Here $M_\l$ denotes the centralizer of $\l$ in $G$. When $G$ is split, $y$ is $L$-permissible means $y$ if and only if in the Iwahori-Weyl group of $L$.

Let $I_L=I \cap L(\LL)$, which is an Iwahori subgroup of $L(\LL)$. For $\tw \in \tW$ we denote by $I_L \tw \s(I_L) /_\s I_L$ (resp. $I \tw I /_\s I$) the set of $I_L$-$\s$-conjugacy classes (resp. $I$-$\s$-conjugacy classes) of $I_L \tw \s(I_L)$ (resp. $I \tw I$). We consider the following map $$\psi_{L, \tw}^G: I_L \tw \s(I_L) /_\s I_L \to I \tw I /_\s I$$ induced by the natural inclusion $I_L \tw \s(I_L) \hookrightarrow I \tw I$. The most interesting case is when $\tw$ is $L$-permissible. Then $\phi_{L,\tw}^G$ is an injection (see Lemma \ref{injection}) but in general not a surjection. However, for {\it $P$-alcove elements} of $\tW$ (see \S \ref{straight}), we have the following affirmative answer. Here $P \subset G$ denotes a {\it semistandard} parabolic subgroup, i.e., a parabolic subgroup containing $T$. We write $P=MN$ to mean that $M \supset T$ is the unique Levi group, and $N$ is the unipotent radical.

\begin{thm}[\cite{GHKR}, \cite{GHN}] \label{N-H}
Let $\tw \in \tW$ be a $P=MN$-alcove. Then

(a) $\psi_{M, \tw}^G$ is a bijection;

(b) If moreover $\tw \s(I_M) \tw\i=I_M$, then $I \tw I$ lies in a single $I$-$\s$-conjugacy class. In this case, we say $\tw$ is $P$-fundamental.
\end{thm}

The theorem was first established in \cite[Theorem 1.1.5]{GHKR} for split groups, where part (a) is referred to as a version of Hodge-Newton decomposition relating affine Deligne-Lusztig varieties associated to $\tw$ for $G$ and those for $L$. Its generalization for quasi-split groups was obtained in \cite[Theorem 3.3.1]{GHN}. In this paper, we adopt the (generalized) notion of $P$-alcove element introduced in loc.cit, but in a slight different form.

\

The first goal of this paper is to find an inverse of Theorem \ref{N-H}.
\begin{thm}\label{bij}
Assume $\tw$ is $L$-permissible. Then $\psi_{L, \tw}^G$ is a bijection if and only if $\tw$ is a $P$-alcove element for some semistandard parabolic subgroup $P=MN$ such that $M \subset L$.
\end{thm}

When $G$ is split, the theorem is an inverse of the Newton-Hodge decomposition result of \cite{GHKR}.

\

Following \cite{GHKR}, we say $\tw \in \tW$ is {\it fundamental} if $I \tw I$ lies in a single $I$-$\s$-conjugacy class. By Theorem \ref{N-H} we know that each $P$-fundamental elements is fundamental. Conversely, we show that each fundamental element arises as a $P$-fundamental element for some semistandard $P$.

\begin{thm}\label{fund}
The following statements on $\tw \in \tW$ are equivalent.

(a) $\tw$ is fundamental;

(b) $\tw$ is $P$-fundamental for some semistandard $P$.

(c) $\tw$ is straight, i.e., $(I \tw\s I)^n=I (\tw\s)^n I$ for all $n \in \NN$.
\end{thm}

$(b) \Rightarrow (a) \& (c)$ is proved in \cite[Proposition 6.3.1 \& 13.3.1]{GHKR}, and $(c) \Rightarrow (a)$ is proved in \cite[Proposition 4.5]{H2}. By this theorem, we can identify fundamental elements with straight elements which have been studied in detail \cite{HN1}.

\

In \cite{VW}, the fundamental elements were used to construct {\it minimal Ekedahl-Oort strata} of good reductions of a PEL type Shimura varieties. By definition, we say $x \in G(\LL)$ is {\it minimal} if $K_1 x K_1$ lies in a single $K$-$\s$-conjugacy class, where $K_1$ is the kernel of the reduction modulo $\e$ map $K=G(\FF[[\e]]) \to G(\FF)$. We obtain a characterization of minimal elements in terms of fundamental elements.

\begin{thm}\label{minele}
An element of $G(\LL)$ is minimal if and only if it lies in a $K$-$\s$-conjugacy class of some fundamental element of $\tW$. In particular, when $G$ is split, each $\s$-conjugacy class of $G(\LL)$ contains one and only one $K$-$\s$-conjugacy class of minimal elements.
\end{thm}
The ``In particular" part can be viewed as a group-theoretic generalization of the main result of \cite{O} on minimal $p$-divisible groups. A special case was proved in \cite[Proposition 4]{VW}.

\

We also obtain a generalization of \cite[Proposition 8.9]{VW}.
\begin{prop}\label{minu}
Let $\mu$ be a minuscule cocharacter. Then each $\s$-conjugacy class intersecting with $K \mu(\e) K$ contains a fundamental element in $W \mu(\e) W$. Here $W=\N_GT(\LL)/T(\LL)$ denotes the (absolute) Weyl group of $G$.
\end{prop}

Combined with \cite[Theorem 9.1]{VW}, Proposition \ref{minu} leads to a generalization of \cite[Theorem 5]{VW} and a direct proof \cite[Theorem 6]{VW}.
\begin{cor}
Let $\ca_0$ be the reduction (modulo $p$) of a Shimura variety defined in \cite{VW}. Then each Newton stratum of $\ca_0$ contains a fundamental Ekedahl-Oort stratum. In particular, all Newton strata of $\ca_0$ are non-empty.
\end{cor}

\

At last, we consider the following natural generalization of fundamental elements. We say $\tw \in \tW$ is {\it $K$-fundamental} (resp. {\it $G(\LL)$-fundamental}) if $I \tw I$ lies in a single $K$-$\s$-conjugacy class (resp. $\s$-conjugacy class). For example, all {\it minimal length elements} in a $\s$-conjugacy class of $\tW$ are $G(\LL)$-fundamental (see \cite[Theorem 3.5]{H2}). We obtain two criteria for $K$-fundamental ($G(\LL)$-fundamental) elements in Theorem \ref{K-f}. To avoid technical details, we just list the one that uses the notion of $P$-alcove element.

\begin{thm}[Theorem \ref{K-f} (c)] \label{Kf}
$\tw \in \tW$ is $K$-fundamental (resp. $G(\LL)$-fundamental) if and only if $\tw$ is a $P_{\nu_{\tw}}=M_{\nu_{\tw}} N_{\nu_{\tw}}$-alcove element and $\tw$ is $M_{\nu_{\tw}}(\FF[[\e]])$-fundamental (resp. $M_{\nu_{\tw}}(\LL)$-fundamental).
\end{thm}
Note that when $P_{\nu_{\tw}}=G$, Theorem \ref{Kf} becomes a null criterion while Theorem \ref{K-f} (b) is still a valid one. See \S\ref{setupred} and \S\ref{setupLevi} for the definition of the semistandard parabolic subgroup $P_{\nu_{\tw}}$.

\

The proofs of our main results of this paper are based on ``reduction method" of Deligne and Lusztig. We show that the notions of fundamental elements, $P$-alcove elements and so on are compatible with this reduction so that the reduction is applicable. Then using some interesting properties of affine Weyl groups (Theorem \ref{min}), we reduce these problems to the cases of some special minimal length elements and verify them directly.

\

{\it Acknowledgement.} The author would like to thank Xuhua He and Chao Zhang for helpful advice and discussions.

\section{Preliminaries}
\subsection{} \label{setuproot} Keep the notations $G$, $T$, $B$ and $\tW$ as in the introduction section. Let $(X^*(T), \Phi, X_*(T), \Phi^\vee)$ be the (absolute) root datum of $G$ with $X_*(T)$ (resp. $X_*(T)$) the character (resp. cocharacter) group and $\Phi$ (resp. $\Phi^\vee$) the set of roots (resp. coroots). The Weyl group $W_0$ acts naturally on the co-characters group $X_*(T)$. Set $V=X_*(T)\otimes_\ZZ \RR$, which admits a $W$-invariant inner product $(,)$ such that $\<\a, \b^\vee\>=\frac{2(\a^\vee, \b^\vee)}{(\a^\vee, \a^\vee)}$ for any $\a, \b \in \Phi$, where $\< ,\>: \Phi \times \Phi^\vee \rightarrow \ZZ$ is the natural pairing, and $\a \leftrightarrow \a^\vee$ denotes the natural correspondence between roots and coroots. Define the norm $||\cdot||: V \to \RR$ by $v \mapsto \sqrt{(v,v)}$. We view the extended affine Weyl group $\tW \cong X_*(T) \rtimes W$ as an isometry subgroup of $V$. We still denote by $\s$ the linear isometry of $V$ induced by the Frobenius automorphism. Let $\overline \cdot: \tW \to W \subset \End(V)$ be the natural projection.

Let $\tPhi$ be the set of {\it real affine roots}, i.e., affine functions on $V$ of the form $a=\a+k: v \mapsto \<\a, v\>+k$ with $\a \in \Phi$ and $k \in \ZZ$. We call the vanishing set $H_a$ of $a$ an {\it affine root hyperplane} and the isometry reflection $s_a \in \tW$ of $V$ along $H_a$ a {\it reflection}. The subgroup $W^a \subset \tW$ generated by all reflections is called the {\it affine Weyl group} of $G$. Note that $\tW$ and $\s$ act naturally on the sets of real affine roots, affine root hyperplanes and reflections. Let $E \subset V$ be a convex subset. We say $e \in E$ is a {\it regular point} of $E$ if for any affine root hyper plane $H$ that contains $e$, one has $E \subset H$. The set of regular points of $E$ is open dense in $E$.

Let $\Phi^+ \subset \Phi$ be the set of (positive) roots occurring in $B$ and $\D=\{v \in V; -1 < \<\a, v\> <0, \a \in \Phi^+\}$ the {\it anti-dominant alcove}. Let $\tPhi^+ \subset \tPhi$ be the set of (positive) real affine roots taking positive values on $-\D$. Define the {\it length function} {\it} $\ell: \tW \to \NN$ by $\ell(\tw)=\sharp \{a \in \tPhi^+; \tw(a)<0\}$. In other words, $\ell(w)$ is the number of affine root hyperplanes which separate $\D$ from $\tw \D$. Let $S^a$ and $S$ be the set of {\it simple reflections}, i.e., reflections of lengths one, of $W^a$ and $W$ respectively. Then both $(W^a, S^a)$ and $(W, S)$ are {\it Coxeter systems}. Since $\s(B)=B$, we have $\s(S)=S$ and $\s(S^a)=S^a$. Let $\Om \subset \tW$ be the subgroup of length zero elements. We have $\tW=W^a \rtimes \Om$.

\subsection{}\label{setupLevi} Let $T \subset M \subset G$ be a semistandard Levi subgroup. With $M$ and $B \cap M$ playing the roles of $G$ and $B$ in \S \ref{setuproot} respectively, we obtain $\tW_M$, $W_M$, $\tPhi_M$, $\tPhi_M^+$, $W^a_M$, $\ell_M$, $S^a_M$, $S_M$, $\D_M$ and so on. If $z \in \tW \rtimes \<\s\>$ preserves $\tPhi_M$, we can also define $\ell_M(z)=\sharp \{a \in \tPhi_M^+; z(a)<0\}$.

Let $v \in V$. Set $\Phi_v=\{\a \in \Phi; \<\a, v\>=0\}$ and $\Phi_{v,+}=\{\a \in \Phi; \<\a, v\> >0\}$. Let $M_v \subset G$ the semistandard Levi subgroup generated by $T$ and the root subgroup $U_{\a} \subset G$ for $\a \in \Phi_v$. Let $N_v \subset G$ be the unipotent subgroup generated by root subgroups $U_{\a}$ with $\a \in \Phi_{v,+}$. Then $P_v=M_v N_v$ is a semistandard parabolic subgroup. For simplicity, we will write $\tW_v=\tW_{M_v}$, $\tPhi_v=\tPhi_{M_v}$, $\ell_v=\ell_{M_v}$ and so on. Note that each semistandard Levi subgroup (resp. parabolic subgroup) is of the form $M_v$ (resp. $P_v$) (resp. $P_v$) for some $v \in V$.

For any two alcoves $C, C'$ we write $C >_{\a} C'$ if $\<\a, C\> > k > \<\a, C'\>$ for some $k \in \ZZ$ and write $C \ge_{\a} C'$ if $C' \ngtr_\a C$. Now we give the definition of $P$-alcove element. We say $\tw \in \tW$ is a $P$-alcove element if there exists $v \in V$ such that $P=P_v$, $\bar \tw\s(v)=v$ and $\tw \D \ge_{\a} \D$ for all $\a \in \Phi_{v,+}$. Let $w \in W$ such that $w(v)$ is dominant and let $J=\{s \in S; s(w(v))=w(v)\}$. Then $\tw$ is a $P_v$-alcove element if and only if $\tw$ is a $(J, w\i, \s)$-alcove element defined in \cite[\S 3.3]{GHN}.

\begin{lem} \label{length}
Let $v \in V$ and $\tw \in \tW$ such that $\tw\s(v)=v$. Let $s \in W^a$ be a reflection such that $\bar s(v)=v$ and $\ell(s \tw)=\ell(\tw)+1$. Then $s \tw$ is a $P_v$-alcove element if and only if so is $\tw$.
\end{lem}
\begin{proof}
We only prove the ``if'' part. The ``only if'' part can be proved similarly.
Let $\a \in \Phi$ such that $\<\a, v\> > 0$. We have to show $s \tw \D \geqslant_{\a} \D$. Assume otherwise, i.e., $\D >_{\a} s \tw \D$. Since $\tw \D \geqslant_{\a} \D$, we have $\tw \D >_{\a} s \tw \D$. By assumption, $\ell(s \tw)=\ell(\tw)+1$, we see that the affine root hyperplane $H$ of $s$ is the unique affine root hyperplane that separates $\D$ from $s \D$ and separates $\tw \D$ from $s \tw \D$. Since $\bar s \neq s_{\a}$, we have $s \D >_{\a} s \tw \D$. Taking the action of $s$ yields $$\D=s(s \D) >_{\bar s(\a)} s(s \tw \D)=\tw \D,$$ which contradicts the assumption that $\tw$ is a $P_v$-alcove element.
\end{proof}

\begin{lem}\cite[Lemma 6.1]{HN2} \label{f}
Let $y \in \tW$ and $v \in V$ such that $\bar y(v)=v$ and $y$ is a $P_v$-alcove element. Let $s \in S$.

(a) if $\ell(y)=\ell(sy\s(s))$, then $sys$ is a $P_{\bar s(v)}$-alcove element;

(b) If $y > sy\s(s)$, then $\bar s(v)=v$. Moreover, both $sy\s(s)$ and $y\s(s)$ are $P_{v}$-alcove elements.
\end{lem}

\subsection{}\label{setupred} Let $J \subset W^a$ and $W_J \subset W^a$ the parabolic subgroup generated by $J$. We denote by $\tW {}^J$ (${}^J \tW$) the set of minimal representatives of  cosets in $W^a / W_J$ ($W_J \backslash W^a$). We set ${}^J \tW {}^{J'}={}^J \tW \cap \tW {}^{J'}$. If $J \subset S$, we denote by $\Phi_J$ the set of roots spanned by $J$.

Let $s \in S^a$ and $\tw, \tw' \in \tW$. We say $\tw \overset s \to_\s \tw'$ if $\tw'=s \tw \s(s)$ and $\ell(\tw) \ge \ell(\tw')$. We say $\tw \to_{J,\s} \tw'$ with $J \subset S^a$ if there is a sequence $\tw=\tw_0, \dots, \tw_r=\tw'$ such that for each $i$ we have $\tw_{i-1} \overset {s_i} \to \tw_i$ for some $s_i \in J$. We say $\tw \approx_{J,\s} \tw'$ if $\tw \to_{J,\s} \tw'$ and $\tw' \to_{J,\s} \tw$. For simplicity, we will use ``$\to_\s$" and``$\approx_\s$" instead of ``$\to_{S^a,\s}$" and ``$\approx_{S^a,\s}$" respectively.

By the proof of \cite[Lemma 4.4]{H1}, we have the following observation
\begin{lem}\label{Bruhat}
Let $\tw, \tw', x \in \tW$ such that $\tw \approx_\s \tw' \geq x$. Then there exists $y$, which is $\s$-conjugate to $x$ by $W^a$, such that $y \leq \tw$ and $\ell(y) \le \ell(x)$.
\end{lem}

For $\tw \in \tW$, let $\nu_{\tw} \in V$ to be the unique element such that $(\tw\s)^n \s^{-n}=n \nu_{\tw} \in X_*(T)$ for some proper nonzero $n \in \ZZ$. We call the unique dominant element $\bar \nu_{\tw}$ of the $W$-orbit of $\nu_{\tw}$ the {\it Newton point} of $\tw$, and call the image $\k(\tw)$ of $\tw$, under the natural projection $\tW \to \Om \to \Om_{\<\s\>}$, the {\it Kottwitz point} of $\tw$.  Here $\Om_{\<\s\>}$ denotes the $\<\s\>$-coinvariant of $\Om$.

Following \cite[\S 2.2]{HN1}, we define $$V_{\tw}=\{v \in V; \tw\s(v)=v+\nu_{\tw}\}=\{\frac{1}{n}\sum_{i=0}^{n-1} (\tw\s)^i(v); v \in V\},$$ which is a nonempty $\tw\s$-stable affine subspace. Note that for any $v \in V$, $\bar \tw\s(v)=v$ if and only if $V_{\tw}=v+V_{\tw}$.

The following two theorems play a fundamental role in the paper.

\begin{thm}\cite[Proposition 2.8]{HN1} \label{min}
Let $\co$ be a $W^a$-conjugacy class of $\tW$ and $\tw \in \co$. Then there exist $\tw'$ of minimal length in $\co$, a subset $J \subset S^a$ and a straight element $x \in {}^J \tW {}^{\s(J)}$ such that

(a) $\tw \to_\s \tw' \in  x W_{\s(J)}$;

(b) The parabolic subgroup $W_J$ and $J=x \s(J) x\i$ and $\tw'=xu$;

(c) The closure $\bar \D$ of $\D$ contains a regular $e$ point of $V_{\tw'}$. So any affine root hyperplane that contains $e$ also contains $V_{\tw'}$.

In particular, we have $\tw' > x$ and $\nu_{\tw'}=\nu_x$. So $x$, $\tw'$ and $\tw$ are in the same $\s$-conjugacy class of $G(\LL)$.
\end{thm}

\begin{thm}\cite[Corollary 2.6]{H1}\label{min'}
Let $\co$ be a $W$-conjugacy class of $\tW$ and $\tw \in \co$. Then there exists $\tw'$ of minimal length in $\co$ such that

(a) $\tw \to_{S,\s} \tw'$;

(b) There exist $J \subset S$, $u \in W_{\s(J)}$ and $x \in {}^S \tW {}^{\s(J)}$ such that $J=x \s(J) x\i$ and $\tw'=xu$;
\end{thm}

\subsection{}\label{Iwahori} Let $I \subset G(\LL)$ be the Iwahori subgroup defined in the introduction. We have $$I=T(\FF[[\e]]) \prod_{a \in \tPhi^+} U_a(\FF),$$ where $U_a$ is the corresponding affine root subgroup of $a$. For $\tw \in \tW$, each $g \in I \tw I$ has a unique expression $g=g_1 g_2$, where $g_1 \in \prod_{a \in \tPhi^+, \tw\i(a)<0}U_a(\FF)$ and $g_2 \in \tw I$.

Let $s \in S^a$. Define $I^s=T(\FF[[\e]])\prod_{a \in \tPhi^+ \setminus \{a_s\}}U_a(\FF)$ and $U_s=U_{a_s}(\FF)$, where $\a_s$ denotes the positive real corresponding to $s$. Then we have the following decomposition $$I=I^s U_s=U_sI^s.$$ Note that both $I^s$ and $U_s$ are subgroups of $I$ and $I^s \cap U_s=\{1\}$. We will write $I_v=I_{L_v}=I \cap L_v(\LL)$ for $v \in V$. The following is a key observation of this paper.
\begin{lem}\label{decrease}
Let $\tw \in \tW$ and $s \in S^a$ such that $s \tw < \tw$ and $\tw \s(s) < \tw$. Then $(I \cdot_\s I^s \tw \s(I^s)) \cap \tw I \subset \tw \s(I^s)$. Here $I \cdot_\s I^s \tw \s(I^s)=\{i\i h \s(i); i \in I, h \in I^s \tw \s(I^s)\}$.
\end{lem}

\subsection{}\label{straight} By the definition in Theorem \ref{fund}, we see that $\tw \in \tW$ is straight if and only if $\ell((\tw\s)^n)=n\ell(\tw\s)=n\ell(\tw)$ for any $n \in \NN$ or equivalently, $\ell(\tw)=\<2\rho, \bar \nu_{\tw}\>$, where $\rho=\frac{1}{2}\sum_{\a \in \Phi^+}\a$. A conjugacy class of $\tW$ is {\it straight} if it contains some straight elements. Any straight element is of minimal length in its conjugacy class. We have the following refinement of Kottwitz's classification of $\s$-conjugacy classes \cite{K}.

\begin{thm}\cite[Theorem 3.7]{H2}\label{str}
The natural inclusion $\N_GT(\LL) \hookrightarrow G(\LL)$ induces a bijection between straight conjugacy classes of $\tW$ and $\s$-conjugacy classes of $G(\LL)$.
\end{thm}
We refer to \cite[Section 3]{HN1} for more on straight conjugacy classes.

\

Throughout the paper, we use the notation ${}^x g=xgx\i$ and ${}^\s g=\s(g)$ for $x, g \in G(\LL)$. In certain circumstance, we will use the same symbol to denote an element of $\tW$ and one of its representatives in $\N_GT(\LL)$. For $X, Y \subset G(\LL)$ we set $X \cdot_\s Y=\{x\i y \s(x); x \in X, y \in Y\}$.

\section{Proof of Theorem\ref{bij}}
In this section we prove Theorem \ref{bij}.

\begin{lem}\label{injection}
Let $v \in V$ and $\tw \in \tW$ with $\bar \tw\s(v)=v$. Then $\psi_{M_v, \tw}^G$ is injective.
\end{lem}
\begin{proof}
The proof is similar to that of the last statement of \cite[Theorem 2.1.2]{GHKR}. The key fact is the following Iwahori decomposition $${}^{\tw\s}I \cap I= ({}^{\tw\s}I_{N_v} \cap I_{N_v}) ({}^{\tw\s}I_{M_v} \cap I_{M_v}) ({}^{\tw\s} I_{N_{-v}} \cap I_{N_{-v}}).$$ Here $I_{N_v}=I \cap N_v(\LL)$ and $I_{N_{-v}}=I \cap N_{-v}(\LL)$.
\end{proof}

\begin{lem}\cite[Lemma 3.3]{GHN} \label{change}
Let $v \in V$ and $s \in S^a$ such that $\bar s(v) \neq v$. Then we have $s \D_v=\D_{\bar s(v)}$. In particular, $s(\Phi_v^+)=\Phi_{\bar s(v)}^+$.
\end{lem}

\begin{lem}\label{f'}
Let $v \in V$ and $\tw \in \tW$ such that $\bar \tw\s(v)=v$. Let $s \in S^a$.

(a) If $\ell(\tw)=\ell(s \tw \s(s))$, then $\psi_{M_v, \tw}^G$ is surjective if so is $\psi_{M_{\bar s(v)}, s \tw \s(s)}^G$.

(b) If $s \tw \s(s) > \tw$ and $\bar s(v)=v$, then $\psi_{M_v, s \tw \s(s)}^G$ is surjective if and only if both $\psi_{M_v, \tw \s(s)}^G$ and $\psi_{M_v, \tw}^G$ are surjective.
\end{lem}
\begin{proof}
For simplicity, we will use the ordinary conjugation action of $G(\LL)$ on $G(\LL)\s \subset G(\LL) \rtimes \<\s\>$ instead of the $\s$-conjugation action of $G(\LL)$ on itself since they are equivalent. Note that ${}^\s I =I$.

(a) Without of loss of generality, we can assume $\tw \neq s \tw \s(s)$ and $s \tw > \tw > \tw \s(s)$. To prove $\psi_{M_v, \tw}^G$ is surjective, it suffices to show that for any $j \in I$, $\tw\s j$ is conjugate by $I$ to some element of $I_v \tw\s I_v$. Conjugating $\tw\s j$ by an appropriate element of $U_s$ we can assume $\tw\s j \in \tw\s I^s$. Then $s\i \tw\s j s \in s \tw \s s I$. Since $\psi_{M_{\bar s(v)}, s \tw \s(s)}^G$ is surjective, there exist $i \in I$ and $i_{\bar s(v)} \in I_{\bar s(v)}$ such that $i\i s\i \tw\s s i_{\bar s(v)} i=s\i \tw\s j s \in s \tw \s s I$. Hence $i \in I^s$ because $\tw\s s < s \tw\s s$. Now we have $\tw\s j={}^s i\i \ \tw\s \ {}^s i_{\bar s(v)} \ {}^{\ds}i \in I \tw\s I$ with ${}^s i \in I$. It remains to show $\tw\s\ {}^s i_{\bar s(v)} \in I_v$. If $\bar s(v) \neq v$, then ${}^s i_{\bar s(v)} \in {}^s I_{\bar s(v)}=I_v$ by Lemma \ref{change}. Otherwise, $i_{\bar s(v)} \in I_v$ and $\tw\s \ {}^s i_{\bar s(v)} \in I_v \tw\s I_v$. Therefore (a) is proved.

(b) ($\Rightarrow$) To prove $\psi_{M_v, \tw \s(s)}^G$ is surjective, it suffices to show that for any $j \in I$, $\tw\s s j$ is conjugate by $I$ to  some element of $I_v \tw\s s I_v$. Conjugating $\tw\s s j$ by an appropriate element of $U_s$, we may assume $j \in I \backslash I^s$. Now $s\i \tw\s s j s \in I^s s\tw\s s I$. Since $\psi_{M_v, s \tw \s(s)}^G$ is surjective, there exist $i_v \in I_v$ and $i \in I$ such that $i\i s\i \tw\s s i_v i=s\i \tw\s s j s \in I^s s\tw\s s I$. Hence $i \in I^s$ because $\tw\s s < s \tw\s s$. Then $\tw\s s j={}^s i\i \ \tw\s s i_v s\i \ {}^s i$ with ${}^s i \in I$. Hence $\tw\s s i_v s\i \in I_v \tw\s s I_v$ as desired.

To prove $\psi_{M_v, \tw}^G$ is surjective, it suffices to show that for any $j \in I$, $\tw\s j$ is conjugate by $I$ to some element of $I_v \tw\s I_v$. Conjugating $\tw\s j$ by an appropriate element of $U_s$, we may assume $j \in I^s$. Then $s\i \tw\s j s \in s\tw\s s I^s$. Since $\psi_{M_v, s \tw \s(s)}^G$ is surjective, there exist $i_v \in I_v$ and $i \in I$ such that $s\i \tw\s j s=i\i s\i \tw\s s i_v i$. Then $i \in I^s$ and $\tw\s j={}^s i\i \ \tw \s s i_v s\i \ {}^s i$ with ${}^{\ds}i \in I$. Hence $\tw\s s i_v s\i \in I_v \tw\s I_v$ as desired.

(b) ($\Leftarrow$) To prove $\psi_{M_v, s \tw \s(s)}^G$ is surjective, it suffices to show that for any $j \in I$, $s\tw\s s\i j$ is conjugate by $I$ to some element of $I_v s\tw\s s I_v$. If $j \notin I^s$, then $\tw\s s\i j s \in I \tw\s s I$. Since $\psi_{M_v, \tw \s(s)}^G$ is surjective, there exist $i_v \in I_v$ and $i \in I$ such that $\tw\s s\i j s=i\i \tw\s s i_v i$. Write $i=u_s i_s$ with $u_s \in U_s$ and $i_s \in I^s$. Then ${}^{s}i_s\i \ s u_s\i \tw\s s i_v u_s s\i \ {}^{s}i_s=s\tw\s s\i j \in I s\tw\s s I$. We derive that ${}^{s}i_s \in I$ and $s u_s\i \tw\s s i_v u_s s\i \in I_v s\tw\s s I_v$ as desired. If $j \in I^s$, then $\tw\s s\i j s \in \tw I^s$. Since $\psi_{M_v, \tw}^G$ is surjective, there exist $i_v \in I_v$ and $i \in I$ such that $\tw\s s\i j s=i\i \tw\s i_v i$. Write again $i=u_s i_s$ with $u_s \in U_s$ and $i_s \in I^s$. Then ${}^{s}i_s\i \ s u_s\i \tw\s i_v u_s s\i \ {}^{s}i_s=s\tw\s s\i j \in I s\tw\s s I$. We derive that $s u_s\i \tw\s i_v u_s s\i \in I_v s\tw\s s I_v$ as desired.
\end{proof}

\begin{proof}[Proof of Theorem \ref{bij}]
$(\Leftarrow)$ By Theorem \ref{N-H} (a), $\psi_{M, \tw}^G$ is a surjection which factors through $\psi_{L, \tw}^G$. Hence so is $\psi_{L, \tw}^G$. The injectivity of $\psi_{L, \tw}^G$ follows from to Lemma \ref{injection}.

$(\Rightarrow)$ We argue by induction on the length of $\tw$. First we show the theorem is true if $\tw$ is of minimal length in its $\s$-conjugacy class. By Theorem \ref{min}, there exists $\tw'$ such that $\tw \approx_\s \tw'$ and the closure $\bar \D$ contains a regular point $e$ of $V_{\tw'}$. Due to Lemma Lemma \ref{f'} (a) and \ref{f} (a), we can assume $\tw=\tw'$. Choose $v \in V$ in a sufficiently small neighborhood of $\nu_{\tw}$ such that $v$ is a regular point of vector subspace spanned by $\l$ and $\nu_{\tw}$. Then we have $\bar \tw\s(v)=v$ (recall that by assumption, $\bar \tw\s(\l)=\l$) and $M_v \subset M_\l=L$. We show that $\tw$ is a $P_v$-alcove element. It remains to prove that for any $\a \in \Phi$ with $\<\a ,v\> >0$, we have $\tw \D \ge_\a \D$.
Assume otherwise. Then there exists $k \in \ZZ$ such that $\<\a, \tw \D\> < k < \<\a, \D\>$. Since $e \in \bar \D$ and $e+\nu_{\tw}=\tw\s(e) \in \tw \bar \D$, we have $\<\a, e+\nu_{\tw}\> \le k \le \<\a, e\>$ and hence $\<\a, \nu_{\tw}\> \le 0$. On the other hand, $\<\a, \nu_{\tw}\> \ge 0$ since $v$ lies in a sufficiently small neighborhood of $\nu_{\tw}$. So $\<\a, \nu_{\tw}\>=0$ and $\<\a, e\>=k$. In other words, $e$ lies on the affine root hyperplane $H$ of $\a-k \in \tPhi$. Since $e$ is a regular point of $V_{\tw}$, we have $V_{\tw} \subset H$. Hence $\<\a, v\>=0$ since $V_{\tw}=v+V_{\tw}$, which is a contradiction.

Now we assume the theorem holds for any element whose length is strictly less than that of $\tw$. We assume $\tw$ is not a minimal length element in its conjugacy class. By Theorem \ref{min}, there exist $\tw_1$ and $s \in S^a$ such that $\tw \approx_\s \tw_1$ and $s \tw_1\s s < \tw_1\s$. By the same reason above, we can assume $\tw=\tw_1$. We show that $\bar s(\l)=\l$. Assume otherwise. Then $I_L=I_{M_\l} \subset I^s$. By Lemma \ref{decrease}, $(I \cdot_\s I_L \tw \s(I_L)) \cap \tw I \subset \tw \s(I^s)$, contradicting the fact that $\psi_{L, \tw}^G$ is surjective. Hence by Lemma \ref{f'} (b), $\psi_{L, \tw \s(s)}^G$ and $\psi_{L, s \tw \s(s)}^G$ are also surjective. By induction hypothesis, there exists $v' \in V$ (resp. $v'' \in V$) such that $M_{v'} \subset L$ (resp. $M_{v''} \subset L$) and $\tw \s(s)$ (resp. $s \tw \s(s)$) is a $P_{v'}$-alcove (resp. $P_{v''}$-alcove) element. We also assume that $\bar \tw \s \bar s(v')=v'$ and $\bar s \bar \tw \s \bar s(v'')=v''$. By \cite[Lemma 4.4.4]{GHN}, either $\bar s(v')=v'$ or $\bar s(v'')=v''$. Thanks to \cite[Lemma 4.4.2]{GHN}, $\tw$ is either a $v'$-alcove element or a $v''$-alcove element as desired.
\end{proof}

\section{Fundamental elements}
In this section, we prove the main results on fundamental elements.

\begin{proof}[Proof of Theorem \ref{fund}]
$(a) \Rightarrow (c)$ Let $\tw \in \tW$ be fundamental. We first show that $\tw$ is of minimal length in its $\s$-conjugacy class. Assume otherwise. By Theorem \ref{min}, there exists $\tw'$ and $s \in S$ such that $\tw \approx_\s \tw'$ and $\tw'\s > s \tw'\s s$. By Lemma \ref{f'}, $\tw'$ is also fundamental. Applying Lemma \ref{decrease}, we see that $I \tw' I$ can not be contained in a single $I$-$\s$-conjugacy class, which is a contradiction.

Now by Theorem \ref{min}, there exist $J \subset S^a$ with $W_J$ finite, $u \in W_{\s(J)}$ and a straight element $x \in {}^J \tW {}^{\s(J)}$ such that $J=x\s(J)x\i$ and $\tw \approx_\s xu$. By Lemma \cite[Lemma 3.1]{H2}, we may assume $\tw=xu$. As in the proof of \cite[Lemma 3.2]{H2}, let $\cp \subset G(\LL)$ be the standard {\it parahoric subgroup} corresponding to $J$ with prounipotent radical $\cu$. Then $\bar \cp=\cp / \cu$ is a reductive group and $\bar I = I / \cu$ is a Borel subgroup of $\bar \cp$. Since $x \in {}^J \tW {}^{\s(J)}$, we have ${}^{x\s} \bar I=\bar I$. By assumption, $\tw=xu$ is fundamental, which means $({}^{xu} I \cap I) \cdot_{\s} (xu) =xu I$. Hence $({}^{xu}\bar I \cap \bar I) \cdot_{\s} (xu) =xu \bar I$. In particular, $\dim {}^{xu}\bar I \cap \bar I \ge \dim \bar I$, which implies $u=1 \in W_{\s(J)}$. Therefore $\tw=x$ is straight as desired.

It remains to show $(c) \Rightarrow (b)$. Let $\tw$ be a straight element. We show that $\tw$ is $P_{\nu_{\tw}}$-fundamental.

By Theorem \cite[Theorem 2.8 \& Proposition 3.2]{HN1}, there exists $\tw'$ such that $\tw \to_\s \tw'$ such that $\D \cap V_{\tw'} \neq \emptyset$. Since $\tw$ is straight, $\tw_1$ is also straight and $\tw \approx_\s \tw'$. We claim that $\tw'$ is $P_{\nu_{\tw'}}$-fundamental. By the same argument in the proof of Theorem \ref{bij}, we see $\tw'$ is a $P_{\nu_{\tw'}}$-alcove element. Choose $e \in \D \cap V_{\tw'} \subset \D_{\nu_{\tw'}}$. Then $\tw\s(e)=e+\nu_{\tw'} \in \D_{\nu_{\tw'}}$, which means $\tw_1 \D_{\nu_{\tw'}}=\D_{\nu_{\tw'}}$ and ${}^{\tw'} I_{\nu_{\tw'}}=I_{\nu_{\tw'}}$. Therefore $\tw'$ is $P_{\nu_{\tw'}}$-fundamental as desired. To complete the inductive, it remains to show the following statement.

\

{\it Let $s \in S^a$ and $\tw, \tw_1 \in \tW$ such that $\tw=s \tw_1 \s(s)$ and $\ell(\tw)=\ell(\tw_1)$. Suppose $\tw_1$ is $P_{\nu_{\tw_1}}$-fundamental, then $\tw$ is $P_{\nu_{\tw}}$-fundamental.}

\

Note that $\nu_{\tw}=\bar s(\nu_{\tw_1})$. Hence by Lemma \ref{f}, $\tw$ is a $P_{\nu_{\tw}}$-alcove element. We show that ${}^{\tw} I_{\nu_{\tw}}=I_{\nu_{\tw}}$. If $\bar s(\nu_{\tw_1}) \neq \nu_{\tw_1}$, it follows from Lemma \ref{change}. Otherwise, $\nu_{\tw}=\nu_{\tw_1}$ and $s \in S^a_{\nu_{\tw_1}}$. Since $\ell(\tw)=\ell(\tw_1)$, we have $\ell_{\nu_{\tw_1}}(\tw)=\ell_{\nu_{\tw_1}}(\tw_1)=0$, where the last equality is due to the assumption that $\tw_1$ is $P_{\nu_{\tw'}}$-fundamental. Therefore $\tw=\tw_1$. The proof is finished.
\end{proof}

To prove Theorem \ref{minele}, we need a result analogous to \cite[Theorem 1.1]{V}.
\begin{lem}\label{K1}
Let $h \in G(\LL)\s$. Then there exist $x \in {}^S \tW$ such that $h$ is conjugate by $K$ to some element of $K_1 x\s K_1$. Moreover, if $h \in I x I$, then $h$ is conjugate by $I$ to some element of $K_1 x K_1$.
\end{lem}
\begin{proof}
The proof consists of two steps.

Step 1. There exists $x \in {}^S\tW$, of minimal length in its $W$-conjugacy class, such that $h$ is conjugate by $K$ to some element $h_1 \in \cu_J x\s \cu_J$. Moreover, if $h \in I x I$, then $h$ is conjugate by $I$ to some element of $\cu_J x\s \cu_J$. Here $J=\max \{J' \subset S; x\s(J')x\i=J'\}$ and $\cu_J=K_1 N_J$, where $N_J=\prod_{\a \in \Phi^+ \setminus \Phi_J} U_\a(\FF)$. This step follows from the proofs of \cite[Lemma 3.1 \& 3.2]{H2} by using Theorem \ref{min'} instead of Theorem \ref{min}.

Step 2. $h_1 \in \cu_J x\s$ is conjugate by $N_J$ into $K_1 x\s K_1$. By \cite[Lemma 4.3]{HN2}, for each $\a \in \Phi^+ \setminus \Phi_J$, there exists $n_\a \in \ZZ_{>0}$ such that ${^{(x\s)^{-n_\a}}} (U_\a(\FF)) \subset K_1$ and ${^{(x\s)^{-i}}} (U_\a(\FF)) \subset N_J$ for $0 \le i < n_\a$. Let $<$ be a linear order on the set roots of $N_J$ such that $\a > \b$ if $\<\a-\b, \rho^\vee\> > 0$, where $\rho^\vee=\frac{1}{2} \sum_{\g \in \Phi^+} \g^\vee$. Set $N_J^{\ge\a}=\prod_{\a \le \b} U_\b(\FF)$ and $N_J^{> \a}=\prod_{\a < \b} U_\b(\FF)$. Assume $h_1 \in K N_J^{\ge \a} x\s \setminus K N_J^{> \a} x\s $. Write $h_1 \in K N_J^{> \a} u_\a x\s $. Then $u_\a \neq 1$. Let $z_\a={^{{(x\s)}^{-n_\a+1}}}u_\a \cdots {^{{(x\s)}^{-2}}}u_\a {^{{(x\s)}^{-1}}}u_\a \in N_J$. Then we have $z_\a h_1 z_\a\i \in K N_J^{> \a} x\s$. Now Step 2 follows by induction on the linear order $>$.
\end{proof}

The following lemma should be well known to experts. We give a proof for completeness.
\begin{lem}\label{J}
Let $x \in {}^S\tW$ and $w \in W$. If $wx\s=x\s w$. Then there exists $J \subset S$ such that $x \s(J) x\i=J$ and $w \in W_J$.
\end{lem}
\begin{proof}
Let $w=s_1 s_2 \cdots s_r$ be a reduce expression, where $s_i \in S$ for $i=1, 2, \dots, r$. Let $\supp(w) \doteq \{s_1, \dots, s_r\}$ be the {\it support} of $w$, which is independent of the choice of reduced expression of $w$. We show that $J=\supp(w)$ meets the requirement of the lemma. Write $x\s s_1 \cdots s_i=w_i x_i\s$ with $w_i \in W$ and $x_i \in {}^S \tW$ for $i=1, \dots, r$. Since $\ell(x\s w)=\ell(wx\s)=\ell(x)+\ell(w)$, $x_{i-1}\s s_i > x_{i-1}\s$. Hence we have either $x_i=x_{i-1}\s(s_i) \in {}^S\tW$ or $x_i=x_{i-1}$ and $s_i' \doteq x_{i-1} \s(s_i) x_{i-1}\i \in S$. In any case we have $x_i \ge x_{i-1}$. Since $x\s w=w_r x_r\s=w x\s$, we have $w_r=w$ and $x_i=x$ for $i=1, \dots, r$. Hence $s_i'=x \s(s_i) x\i \in \supp(w_r)=\supp(w)=J$, which implies $J=x J x\i$ as desired.
\end{proof}

\begin{proof}[Proof of Theorem \ref{minele}]
Let $g \in G(\LL)$ be a minimal element. By Lemma \ref{K1}, $g \in K \cdot_\s K_1 x K_1$ for some $x \in {}^S \tW$. We show that $x$ is fundamental. Since $g$ is minimal, $x$ is also minimal. Therefore $$I x I \subset I \cdot_\s K_1 x K_1 \subset K \cdot_\s x,$$ where the first inclusion follows from the ``Moreover" part of Lemma \ref{K1}. In other words, $x$ is $K$-fundamental and the first statement of the theorem follows form Lemma \ref{KI} below. The ``In particular" part now follows from Theorem \ref{fund} and \cite[Lemma 6.11]{V}.
\end{proof}

\begin{lem}\label{KI}
Let $x \in {}^S \tW$. Then $x$ is $K$-fundamental if and only if $x$ is fundamental.
\end{lem}
\begin{proof}
The ``if'' part is trivial. We show the ``only if'' part. Let $g \in I u I \subset K$ with $u \in S$ such that $g\i x\s g \in I x\s I$. Since $x\s \in {}^S \tW$, we have $\ell(u\i x\s u)=\ell(u\i x\s)-\ell(u)=\ell(x)$ and $u\i x \s(u)=x \in \tW$. Applying Lemma \ref{J}, we have $x\s(J)x\i=J$, where $J=\supp(u)$. Write $g=h u i$, where $h \in \prod_{\a \in \Phi_J^+, u\i(\a)<0}U_{\a}(\FF)$ and $i \in I$. Then $$g\i x\s g=i\i u\i x\s\ {^{(x\s)\i}}h\i h\ u i \in I x\s I.$$ By the relations between $x\s$, $u$ and $J$, we have $${^{(x\s)\i}}h\i h \in \prod_{\a \in \Phi_J^+, u\i(\a)<0}U_{\a}(\FF).$$ Moreover, by $\ell(x\s)=\ell(u\i x\s)-\ell(u)$, we derive that ${^{(x\s)\i}}h\i h=1$, which implies that $g\i x \s(g) \in I \cdot_\s x$. Therefore $x \in {}^S \tW$ is fundamental if it is $K$-fundamental.
\end{proof}

\begin{proof}[Proof of Proposition \ref{minu}]
Let $C$ be a $\s$-conjugacy class of $G(\LL)$ that intersects with $K \mu(\e) K$. Assume $C \cap I \tw I \neq \emptyset$ for some $\tw \in W \mu(\e) W$. We claim that there exist a straight element $x \in C$ such that $\tw \geq x$ in the sense of Bruhat order. Then the theorem follows in the same way of \cite[Proposition 8.9]{VW}. We prove the claim by induction on the length of $\tw$.

Assume $\tw$ is of minimal length in its $\s$-conjugacy class in $\tW$. Then $\tw \in C$ by \cite[Theorem 3.5]{H2}. Applying Theorem \ref{min}, we have $\tw \approx_\s \tw' \geq x$, where $x \in C \cap \tW$ is a straight element. Therefore, by Lemma \ref{Bruhat}, there exists $y \in \tW$ in the same $\s$-conjugacy class of $x$ which satisfies $\tw \geq y$ and $\ell(y) \le \ell(x)$, which means $y \in C$ is straight since so is $x$ (see \S \ref{straight}).

If $\tw$ is not of minimal length in its $\s$-conjugacy class in $\tW$. By Theorem \ref{min}, there exist $\tw_1 \in \tW$ and $s \in S^a$ such that $\tw \approx_\s \tw_1$ and $\tw_1 > s \tw_1 \s(s)$. Then $C$ intersects with either $I s \tw_1 I$ or $I s \tw_1 \s(s)I$. Since $\tw_1 > s \tw_1, s \tw_1 \s(s)$, the claim holds for $\tw_1$ by induction hypothesis, and hence holds for $\tw$ by the same argument in the above paragraph.
\end{proof}

\section{$K$-fundamental ($G(\LL)$-fundamental) elements}
\begin{lem} \label{red}
Let $\tw \in \tW$ be a $K$-fundamental ($G(\LL)$-fundamental) and $s \in S$ ($s \in S^a$).

(a) If $\ell(\tw)=\ell(s \tw \s(s))$, then $s \tw \s(s)$ is $K$-fundamental ($G(\LL)$-fundamental).

(b) If $\tw > s \tw \s(s)$, then both $\tw \s(s)$ and $s \tw \s(s)$ are $K$-fundamental ($G(\LL)$-fundamental). Moreover, $\bar s(\nu_{\tw})=\nu_{\tw}$ and $\bar s(\nu_{\tw \s(s)})=\nu_{\tw \s(s)}$.
\end{lem}
\begin{proof}
We only prove (b). (a) can be proved similarly. Let $a \in \tPhi^+$ be the simple real affine root corresponding to $s$, i.e., $s_a=s$.

To show $\tw \s(s)$ and $s \tw \s(s)$ are $K$-fundamental (resp. $G(\LL)$-fundamental), it suffices to show that any $h \in I \tw\s s$ and any $h' \in I s\tw\s s$ can be conjugate by $K$ (resp. $G(\LL)$) to some elements of $I \tw\s I$. Note that $\tw\s s(a), s \tw\s s(a)$ are positive real root different from $a$. We can $\s$-conjugate $h$ and $h'$ by suitable elements of $U_s$ respectively so that $h \in I \tw\s s \setminus I^s \tw\s s$ and $h' \in I^s s\tw\s s$. Then we have $s\i h s, s\i h' s \in I \tw\s I$ as desired.

Now we come to the ``Moreover" part. Let $H=H_a$ be the affine root hyperplane of $a$. If $\bar s(\nu_{\tw}) \neq \nu_{\tw}$, then $H \cap V_{\tw} \neq \emptyset$ because $V_{\tw}=\tw\s V_{\tw}=V_{\tw}+\nu_{\tw}$. Choose $p \in H \cap V_{\tw}$. By \cite[Lemma 2.3]{HN1}, we have $$||\nu_{\tw \s(s)}|| \le ||\tw\s s(p)-p||=||\tw\s(p)-p||=||\nu_{\tw}||$$ and the equality holds if and only if $p \in V_{\tw \s(s)}$. On the other hand, we already show that $\tw\s(s)$ and $\tw$ are in the same $\s$-conjugate class in $G(\LL)$. Hence $\nu_{\tw \s(s)}$ and $\nu_{\tw}$ are conjugate by $W$ and hence $||\nu_{\tw \s(s)}||=||\nu_{\tw}||$. Therefore $p \in V_{\tw \s(s)}$ and $$\bar \tw \s \bar s (\nu_{\tw \s(s)})=\nu_{\tw \s(s)}=\tw\s s(p)-p=\tw\s(p)-p=\nu_{\tw}=\bar \tw\s(\nu_{\tw}),$$ which is a contradiction to the assumption $\bar s(\nu_{\tw}) \neq \nu_{\tw}$. Exchanging the roles of $\tw$ and $\tw \s(s)$ in the above argument yields $\bar s(\nu_{\tw s})=\nu_{\tw s}$.
\end{proof}

\begin{thm}\label{K-f}
Let $\tw \in \tW$. Then the following statements are equivalent:

(a) $\tw$ is $K$-fundamental ($G(\LL)$-fundamental).

(b) Write $\tw\s \in x\s W^a_{\nu_{\tw}}$ with $\ell_{\nu_{\tw}}(x\s)=0$. Then $x$ is fundamental and there exists $J \subset S_{\nu_{\tw}}$ ($J \subset S^a_{\nu_{\tw}}$ with $W_J$ finite) such that $\tw \in x\s W_J$, $x \s(J) x\i=J$ and $\ell(\tw\s)=\ell(x\s)+\ell_{\nu_{\tw}}(\tw\s)$. In particular, $\nu_{\tw}=\nu_x$.

(c) $\tw$ is both a $P_{\nu_{\tw}}$-alcove element and a $M_{\nu_{\tw}}(\FF[[\e]])$-fundamental ($M_{\nu_{\tw}}(\LL)$-fundamental) element.
\end{thm}

\begin{rmk} It is worth mentioning that in Theorem \ref{K-f} (b), $x$ is uniquely determined by $\tw$ by requiring that $\tw\s \in x\s W^a_{\nu_{\tw}}$ with $\ell_{\nu_{\tw}}(x\s)=0$. Moreover, we always have $\ell(\tw\s) \ge \ell(x\s)+\ell_{\nu_{\tw}}(\tw\s)$. \end{rmk}

\begin{proof}
We only give a proof for the $K$-fundamental case. The $G(\LL)$-fundamental case can be handled similarly by using Theorem \ref{min} instead of Theorem \ref{min'}.

(a) $\Rightarrow$ (b) We argue by induction on the length $\tw$. If $\tw$ is of length zero, it is trivial. We assume the statement is true for any $K$-fundamental element with strictly smaller length than the length of $\tw$.

By Theorem \ref{min'}, we have $\tw \to_{S,\s} \tw' \in x' W_{\s(J')}$ for some $x' \in {}^S \tW$ and $J' \subset S$ such that $x' \s(J') {x'}\i=J'$. Hence $\nu_{x'}=\nu_{\tw'}$ and $J' \subset S_{\nu_{\tw'}}$. By Lemma \ref{red}, $\tw'$ is $K$-fundamental. Similar to \cite[Lemma 3.2]{H2}, we have $K \cdot_\s I \tw' I=K \cdot_\s I x' I$. Hence $x'$ is also $K$-fundamental. By Lemma \ref{KI}, $x'$ is fundamental. Hence $\ell_{\nu_{x'}}(x'\s)=0$ since $x'$ is also $P_{\nu_{x'}}$-fundamental by the proof of Theorem \ref{fund}. Obviously, we have $\ell(\tw'\s)=\ell(x'\s)+\ell_{\nu_{\tw'}}(\tw'\s)$. Therefore $\tw'$ satisfies (b). Thanks to Lemma \ref{red}, to show $\tw$ satisfies (b), we can assume that there exists $s \in S$ such that $\tw \overset s \to_{\s} \tw_1=s \tw \s(s)$ and $\tw_1$ satisfies (b). There are two cases as follows.

Case 1 $\tw > \tw_1$. By Lemma \ref{red}, $\tw \s(s)$ is $K$-fundamental and $s(\nu_{\tw \s(s)})=\nu_{\tw \s(s)}$. Then $\tw \s(s)$ satisfies (b) by induction hypothesis. Hence $\tw\s s \in x_2\s W_{J_2}$ for some fundamental element $x_2$ with $\ell_{\nu_{\tw \s(s)}}(x_2\s)=0$ and $J_2 \subset S_{\nu_{\tw \s(s)}}$ such that $x_2 \s(J_2) x_2\i=J_2$ and $\ell(\tw\s s)=\ell(x_2\s)+\ell_{\nu_{\tw \s(s)}}(\tw\s s)$. Since $s \in S_{\nu_{\tw \s(s)}}$ and $s \tw\s s < \tw\s s$, we have $s \in J_2$ and $\tw\s \in x_2\s W_{J_2}$, which means $\nu_{\tw}=\nu_{\tw \s(s)}=\nu_{x_2}$ and $x=x_2$.
It remains to show that $\ell_{\nu_{\tw}}(\tw\s)=\ell_{\nu_{\tw}}(\tw\s s)+1$, which follows from $\tw\s s < \tw\s$ and $s \in S_{\nu_{\tw}}$. Therefore $\tw$ satisfies (b).

Case 2 $\ell(\tw)=\ell(\tw_1)$ and $\tw \neq \tw_1$. By assumption, $\tw_1 \in x_1\s W_{J_1}$ for some fundamental element $x_1$ with $\ell_{\nu_{\tw_1}}(x_1\s)=0$ and $J_1 \subset S_{\nu_{\tw_1}}$ such that $x_1 \s(J_1) x_1\i=J_1$ and $\ell(\tw_1\s)=\ell(x_1\s)+\ell_{\nu_{\tw_1}}(\tw_1\s)$. First we assume $s \in S_{\nu_{\tw_1}}$, then $x=x_1$. Since $x \s(J_1) x\i =J_1$ and either $s \tw_1\s < \tw_1\s$ or $\tw_1\s s < \tw_1\s$, we have $s \in J_1$ and $\tw\s \in x\s W_{J_1}$. It remains to show $\ell_{\nu_{\tw}}(\tw\s)=\ell_{\nu_{\tw}}(\tw_1\s)$ which follows from the equality $\ell(\tw\s)=\ell(\tw_1\s)$. Now we assume $s(\nu_{\tw_1}) \neq \nu_{\tw_1}$. By Lemma \ref{change}, we have $x=s x_1 \s(s)$ and $\ell_{\nu_{\tw}}(x\s)=0$. Let $J_1'=s J_1 s \subset S_{\nu_x}$ and $\ell_{\nu_{\tw_1}}(\tw_1\s)=\ell_{\nu_{\tw}}(\tw\s)$. Moreover we have $x \s(J_1') x\i=J_1'$. Then $$\ell(x_1\s)+\ell_{\nu_{\tw_1}}(\tw_1\s)=\ell(\tw_1\s)=\ell(\tw\s) \ge \ell(x\s)+\ell_{\nu_{\tw}}(\tw\s),$$ which means $\ell(x_1\s) \ge \ell(x\s)$. Since $x_1$ is fundamental (straight), $x=s x_1 \s(s)$ is also fundamental and $\ell(x)=\ell(x_1)$. In a word, $\tw$ satisfies (b).

(b) $\Rightarrow$ (c) By the proof of \cite[Lemma 3.2]{H2}, $\tw$ is $M_{\nu_{\tw}}(\FF[[\e]])$-fundamental. Since $x$ is fundamental, $x$ is a $P_{\nu_x}$-alcove element, hence a $P_{\nu_{\tw}}$-alcove element. By Lemma \ref{length}, we see that $\tw$ is a $P_{\nu_{\tw}}$-alcove element.

(c) $\Rightarrow$ (a) It follows from Theorem \ref{bij}.
\end{proof}


\begin{thebibliography}{GKP00}
\bibitem[GHKR]{GHKR}
U.~Goertz, T.~Haines, R.~Kottwitz, D.~Reuman, \emph{Affine Deligne-Lusztig va-
rieties in affine flag varieties}, Compositio Math. 146 (2010), 1339-1382.

\bibitem[GHN]{GHN}
U.~Goertz, X.~He, S.~Nie, \emph{$P$-alcoves and non-emptiness of affine Deligne-lusztig varieties}, arXiv:1211.3784.

\bibitem[H1]{H1}
X.~ He,  \emph{Minimal length elements in some double cosets of {C}oxeter groups}, Adv. Math. \textbf{215} (2007), no.~2, 469--503.

\bibitem[H2]{H2}
X.~He, \emph{Geometric and homological properties of affine Deligne-Lusztig varieties}, Ann. Math. 179 (2014), 367--404.

\bibitem[HN1]{HN1}
X.~He and S.~Nie, \emph{Minimal length elements of extended affine Weyl groups, II}, arXiv:1112.0824, to appear in Compositio Math.

\bibitem[HN2]{HN2}
X.~He and S.~Nie, \emph{$P$-alcoves, parabolic subalgebras and cocenter of affine Hecke algebras}, arXiv:1310.3940.

\bibitem[K]{K}
R.~Kottwitz, \emph{Isocrystals with additional structure}, Compositio Math. \textbf{56} (1985), 201--220.

\bibitem[O]{O}
F.~Oort, \emph{Minimal $p$-divisible groups}, Ann. of Math. \textbf{161} (2005), 1021--1036.


\bibitem[V]{V}
E.~Viehmann, \emph{Truncation of level 1 of elements in the loop group of a reductive group}, arXiv:0907.2331.

\bibitem[VW]{VW}
E.~Viehmann and T.~Wedhorn \emph{Ekedahl-Oort and Newton strata for Shimura varieties of PEL type}, Math. Ann. \textbf{356} (2013), 1493--1550.

\end{thebibliography}
\end{document}